\documentclass[12pt]{amsart}
\usepackage[english]{babel}
\usepackage{amsfonts,amssymb,latexsym,amscd}

\usepackage{graphics}
\usepackage[all,cmtip]{xy} 

\theoremstyle{plain}
\newtheorem{theorem}{Theorem}

\newtheorem{proposition}{Proposition}

\theoremstyle{definition}
\newtheorem{definition}{Definition}
\theoremstyle{remark}

\newtheorem{example}{Example}

\usepackage{hyperref}

\begin{document}
\title{On transitive and homogeneous binary $G$-spaces}

\author[Pavel Gevorgyan, Quitzeh Morales]{Pavel S. Gevorgyan, Quitzeh Morales Melendez}

\address{Moscow State Pedagogical University, Russia} \email{pgev@yandex.ru}

\address{CONACYT -- Universidad Pedag\'ogica Nacional -- Unidad 201 Oaxaca,
Camino a la Zanjita S/N, Col. Noche Buena, Santa Cruz Xoxocotl\'an, Oaxaca.
C.P. 71230}\email{qmoralesme@conacyt.mx}



\thanks{The second author was partially supported by Catedras CONACyT Project 1522}

\begin{abstract}
In this paper, the notions of transitivity and homogeneity in binary $G$-spaces are studied. These notions coincide for distributive binary $G$-spaces. For compact $G$, it is shown that distributive transitive binary $G$-spaces are coset spaces with a suitably defined binary $G$-action. Homogeneous binary $G$-spaces are topologically homogeneous and are separated into distinct stabilization types. Examples of each type are constructed.
\end{abstract}

\maketitle


\section*{Introduction}

The study of homogeneous $G$-spaces is the study of $G$-orbit types. This study has allowed a detailed description of the topological structure of such spaces, both in general and in particular important cases, as those of free or proper $G$-actions. In particular, it has made possible the construction of universal $G$-spaces for many classes of $G$-spaces.

Orbits of binary $G$-spaces were studied in \cite{Gevorkyan2021}, where it was pointed out that usual notions for $G$-spaces, such as orbits, do not easily translate to binary $G$-spaces. In binary $G$-spaces, orbits may intersect. However, it was shown in \cite{Gevorkyan2022} that, in the special case of distributive binary $G$-spaces, orbits either coincide or have an empty intersection. Orbits can be either finitely or infinitely generated. In the case of distributive binary $G$-spaces, the orbits are finitely generated.

Transitive and homogeneous binary $G$-spaces are introduced. For distributive binary $G$-spaces these concepts coincide. A classification result is given in the case of distributive transitive binary $G$-spaces for compact $G$. These are coset spaces by closed normal subgroups with a binary action of $G$ by left multiplication twisted by group conjugation. As a consequence, a classification of free transitive distributive binary G-spaces for compact $G$ is given.

Homogeneous binary $G$-spaces are topologically homogeneous spaces. The stabilization properties of homogeneous binary $G$-spaces are important for the study of these spaces. A homogeneous binary $G$-space may have different stabilization properties at different points. The stabilization properties separate the class of homogeneous binary G-spaces by types taking values in the natural numbers or at $\infty$. Transitive binary $G$-spaces for compact $G$ are examples of the first type. In the paper \cite{Gevorkyan2021}, an example of an infinitely generated binary $G$-space is constructed. Similar constructions give examples of discrete binary $G$-spaces of some types. An application of hyperspherical coordinates on real Euclidean $n$-dimensional spaces gives examples of homogeneous binary $G$-spaces of each finite type for non-discrete, abelian topological group $G$. 


\section{Preliminaries}
In the following, $G$ and $H$ denote topological groups, and $X$ and $Y$ denote topological spaces. All maps are assumed to be continuous.

Recall that an \textit{action} of the group $G$ on the space $X$ is a map
$$\alpha: G\times X \longrightarrow X$$
such that 
$$\alpha(gh,x)=\alpha(g,\alpha(h,x)),$$
$$\alpha(e,x)=x$$
for any $g,h\in G$ and $x\in X$, where $e$ is the identity of $G$. The space $X$, together with a given action $\alpha$ of $G$, is called a $G$-space.

It is customary to omit the map $\alpha$ from the notation and use the notation $gx$ for $\alpha(g,x)$, so that the above identities become $g(hx)= (gh)x$ and $ex = x$.

A map $\varphi:X\longrightarrow Y$ between $G$-spaces $(G,X,\alpha)$ and $(G,Y,\beta)$ is called \textit{equivariant} if $\varphi (\alpha(g,x))=\beta(g,\varphi (x))$ or, for simplicity, $\varphi (gx)=g\varphi (x)$ for any $g\in G, x\in X$.

If $X$ is a G-space and $x \in X$, then the subspace $Gx=\{gx; \ g\in G, \ x\in X\}\subset X$ is called the \textit{orbit} or \textit{$G$-orbit} of $x\in X$. 

A $G$-space $X$ is said to be \textit{transitive} if $Gx=X$ for any $x\in X$. It is well known that if $G$ is a compact group, then any transitive $G$-space is equivariantly homeomorphic to some coset space $G|H$ of a topological group $G$ by a closed subgroup $H$ together with the action of $G$ by left translation: $g(g'H) = (gg')H$ for any $g,g'\in G$.

A map $\mu:G\times X^2 \longrightarrow X$ is called a \textit{binary action} of the topological group $G$ on the space $X$ if the identities
\[
    \mu(gh,x,y)=\mu(g,x,\mu(h,x,y)),
\]
\[
    \mu(e,x,y)=y
\]
are satisfied for any $g,h\in G$ and $x,y\in X$.  In this case the triple
$(G,X,\mu)$ is called a \textit{binary $G$-space}.

By analogy with the case of usual $G$-action we denote $\mu(g,x,y)$ simply by $g(x,y)$. With these notation the identities for a binary action take the form
\[
    gh(x,y)=g(x,h(x,y)),
\]
\[
e(x,y)=y.
\]

A map $\varphi :X\to Y$ between binary $G$-spaces $(G,X,\mu)$ and $(G,Y,\nu)$ is said to be \textit{biequivariant} if the following diagram is commutative:
$$
\xymatrix{ 
G \times X\times X \ar[r]^{1\times \varphi\times \varphi}\ar[d]_\mu &G \times Y\times Y\ar[d]^\nu \\
 X\ar[r]^{\varphi} & Y
 } 
$$
or, equivalently, if $\varphi(g(x,x')) = g(\varphi(x),\varphi(x'))$ for any $g\in G$ and $x,x'\in X$.

The map $\varphi : X\to Y$ is called a \textit{biequimorphism}, if it is a biequivariant homeomorphism with biequivariant inverse.

For a subset  $K\subset G$ of the group $G$ and a subset $A\subset X$ of the binary $G$-space $X$ denote
$$
K(A,A)= \{g(a_1,a_2); \ g\in K, \ a_1,a_2\in A\}.
$$ Considering this, denote $g(A,A)=\{g\}(A,A)$ and $G(x,y)=G(\{x\},\{y\})$.

A subset $A\subset X$ of the binary $G$-space $X$ is said to be \textit{$G$-bi-invariant} or just \textit{bi-invariant} if $G(A,A)= A$.

If $X$ is a binary $G$ space and $x\in X$, then the minimal bi-invariant subset of $X$ containing $x$ is called the \textit{orbit} of the point $x$, which we denote by $[x]$.

The set $G_{(x,x)}= \{g\in G; \ g(x,x)=x\}$ is a subgroup of the group $G$ which we call the \textit{stationary subgroup} or \textit{isotropy subgroup} of the point $x\in X$. 

A binary $G$-space is called \textit{distributive} if the equation
\begin{equation}
    g(h(x,x'),h(x,x''))=h(x,g(x',x''))
\end{equation}
is true for any points $x,x',x''\in X$ and any elements $g,h\in G$.

In \cite{Gevorkyan2021} it is shown that for distributive $G$-spaces one has $[x]=G(x,x)$.

These definitions, as well as all other definitions, notions, and results used in the paper without reference, can be found in \cite{Bredon}--\cite{movsisyan}.


\section{Transitive binary \texorpdfstring{$G$}{G}-spaces}

Consider a topological group $G$ and a topological space $X$.

\begin{definition}
A binary action of the group $G$ on the space $X$ is called \textit{transitive}, if the condition $G(x,x)=X$ is true for any $x\in X$, i.e., if there is precisely one orbit, $X$ itself.  In this case, $X$ is called a \textit{transitive}
binary $G$-space.
\end{definition}

\begin{definition}
A binary action of the group $G$ on 
the space $X$ is called \textit{free}, if for any $x\in X$ the stationary group $G_{(x,x)}$ is trivial. In this case, the space $X$ is called a \textit{free} binary $G$-space.
\end{definition}

\begin{proposition}\label{prop_0}
Let $G$ be a compact group and let $H\subset G$ be a closed normal subgroup. Then a coset space $G|H$ together with the binary action
$\mu : G\times G|H\times G|H \to G|H$ of $G$, defined by the formula
\begin{equation}\label{eq-action}
\mu (g,g_1H,g_2H) = g_1gg_1^{-1}g_2H, \quad \text{or} \quad g(g_1H,g_2H) = g_1gg_1^{-1}g_2H
\end{equation}
for any elements $g,g_1,g_2\in G$, is a transitive binary $G$-space.
\end{proposition}

\begin{proof}
Since $H$ is a normal subgroup of $G$, it is easy to prove that $\mu$ is a well-defined map.

Note that $\mu$ is indeed a binary action of group $G$ on $G|H$: one has
$e(g_1H,g_2H)=g_2H$ 
and
\begin{multline*}
gg'(g_1H,g_2H) = g_1gg'g_1^{-1}g_2H = g_1gg_1^{-1}g_1g'g_1^{-1}g_2H = \\
=g(g_1H, g_1g'g_1^{-1}g_2H) = g(g_1H, g'(g_1H,g_2H)).
\end{multline*}

It is evident that the binary action $\mu$ is transitive. 
\end{proof}

\begin{proposition}
Let $G$ be a compact group and let $H$ and $K$ be closed normal subgroups of $G$. Then there exists a biequivariant map $G|H \to G|K$ between transitive binary $G$-spaces $G|H$ and $G|K$ iff $H$ is a subgroup of $K$.
\end{proposition}

\begin{proof}
Let $H$ be a subgroup of $K$. Define a map $f:G|H \to G|K$ by the formula $f(gH)=gK$. Note that $f$ is well defined, i.e., if $gH=g'H$ then $f(gH)=f(g'H)$. Indeed, it follows from $gH=g'H$ that $g^{-1}g'\in H\subset K$,  $g^{-1}g'K=K$, $g'K=gK$, and hence $f(gH)=f(g'H)$.

The map $f$ is biequivariant:
\begin{multline*}
    f(g(g_1H,g_2H))=f(g_1gg_1^{-1}g_2H) = g_1gg_1^{-1}g_2K = \\ =g(g_1K,g_2K)= g(f(g_1H),f(g_2H)).
\end{multline*}

Conversely, let $f : G|H \to G|K$ be any  biequivariant map and suppose that $f(H)= a K$ for some $a\in G$. It follows from the biequivariance of the map $f$ that for any $h\in H$ one has $f(h(H,H))= h(f(H), f(H)) = h(aK, aK) = aha^{-1}aK = ahK$. On the other hand, $f(h(H,H))=f(hH)=f(H)=aK$. So, $ahK=aK$, $hK=K$. Therefore, $h\in K$ and hence $H\subset K$. 
\end{proof}

\begin{theorem}\label{th_1}
Let $G$ be a compact group. Then any transitive distributive binary $G$-space $X$ is biequimorphic to the binary $G$-space $(G,G|H,\mu)$, where  $H$ is some normal subgroup of $G$.
\end{theorem}

\begin{proof}
Choose any point $x\in X$ and consider the isotropy subgroup $H=G_{(x,x)}$ of the point $x$. For any $h\in H$ and $g\in G$ it follows from the distributivity of the binary action that \begin{multline*}
   ghg^{-1}(x,x)= gh(x,g^{-1}(x,x)) = g^{-1}(gh(x,x), gh(x,x))= \\
   =g^{-1}(g(x,h(x,x)), g(x,h(x,x)))= g^{-1}(g(x,x), g(x,x))= \\
   =g(x, g^{-1}(x,x)) = e(x,x) = x.
\end{multline*} 
So, $gHg^{-1}=H$, i.e., $H$ is a normal subgroup of $G$. 

Now define
$$\varphi:G|H \to X$$
by the formula
\[
\varphi(gH) = \varphi(gG_{(x,x)}) = g(x,x),
\]
where $g\in G$, $x\in X$. 

Note that the map $\varphi$
is bijective. It is also continuous by the definition of the coset space $G|H$ and the continuity of the map $g\mapsto g(x,x)$.
As the space  $G|G_{(x,x)}$ is compact, the map $\varphi$ is a homeomorphism.

It remains to show that $\varphi$ is a biequivariant map, i.e. the equation
\[
\varphi(g(g_1H, g_2H)) = g(\varphi(g_1H), \varphi(g_2H))
\]
holds.

Using the distributivity of the binary action of $G$ on $X$, we obtain
\begin{multline*}
    \varphi(g(g_1H), g_2H)) = \varphi(g_1gg_1^{-1}g_2H) = g_1gg_1^{-1}g_2(x,x) =  \\
    = g_1(x, gg_1^{-1}g_2(x,x)) = g_1(x,g(x,g_1^{-1}g_2(x,x))) = g(g_1(x,x),g_1(x,g_1^{-1}g_2(x,x)))= \\ g(g_1(x,x),g_2(x,x)) = g(\varphi(g_1H), \varphi(g_2H)).
\end{multline*}
for every $g,g_1,g_2\in G$, $x\in X$. So, $\varphi$ is a biequimorphism. 
\end{proof}

The next theorem follows from Theorem \ref{th_1} and Proposition \ref{prop_0}.

\begin{theorem}\label{cor-free-transitive}
Let $G$ be a compact group. 
Then any free transitive distributive binary $G$-space is biequimorphic to the binary $G$-space
$(G,G,\eta)$ with the binary action
\begin{equation}\label{standard_distributive}
  \eta(g,g_1,g_2)=g_1gg_1^{-1}g_2,
\end{equation}
where $g,g_1,g_2\in G$. 
\end{theorem}


\section{Homogeneous binary \texorpdfstring{$G$}{G}-spaces}

\begin{definition}\label{def_homogeneous}
Let $X$ be a binary $G$-space. If there is a point $x\in X$ such that $[x]=X$, then the space $X$ is called a \textit{homogeneous} binary $G$-space. In this case $x$ is called a \textit{stabilization point} of $X$.
\end{definition}

The following example shows that if $X$ is a homogeneous binary $G$-space, then it is not necessarily true that any point $x\in X$ is a stabilization point of $X$.
 
\begin{example}\label{ex-ZZZ}
The binary $G$-space $(G, X, \mu)$ where $G= \mathbb{Z}$, $X= \mathbb{Z}$ and the binary action 
$\mu: G\times X\times X \to X$ is defined by the rule
\begin{equation}
 \mu(n,x,x')= n(x,x')=nx+x'.
\end{equation} is homogeneous. Note that $x=1$ is a stabilization point: $[1]=G(1,1)=X$. However, $x=0$ is not a stabilization point because    
$[0]=G(0,0)=\{0\}\neq X$.  
\end{example}

Notice that transitive binary $G$-spaces are homogeneous. The converse is not true. However, for distributive binary $G$-spaces one has the following.

\begin{proposition}\label{prop_1}
Any homogeneous distributive binary $G$-space is transitive.
\end{proposition}

\begin{proof} 
Consider a homogeneous distributive binary $G$-space $X$. By Definition \ref{def_homogeneous} there is a point $x_0\in X$ such that $[x_0]=X$. As the binary $G$-space $X$ is distributive, one has that $[x_0]=G(x_0,x_0)$, and for any point $x\in X$, $G(x,x)=[x]=[x_0]=G(x_0,x_0)=X$ by \cite[Prop. 2]{Gevorkyan2022}. Hence $X$ is a transitive binary $G$-space.

\end{proof}

\begin{theorem}
Homogeneous binary $G$-spaces are topologically homogeneous spaces. 
\end{theorem}

\begin{proof}
Let $X$ be a homogeneous binary $G$-space. By Definition \ref{def_homogeneous} there exists a point $x_0\in X$ such that
$[x_0]=X$. 

In order to prove that the space $X$ is topologically homogeneous, it is enough to show that  for any point $x^*\in X$ there exists a homeomorphism $\varphi:X \to X$ such that $\varphi(x_0)=x^*$.

Consider the following sequence of subsets of the binary $G$-space $X$:
\begin{equation}\label{eq-G^n}
G^1(x_0)=G(x_0,x_0), \, \ldots\,, \  G^{n}(x_0)=G(G^{n-1}(x_0), G^{n-1}(x_0)), \, \ldots
\end{equation}
where $n=1, 2, \ldots $ 

It follows from \cite[Proposition 7]{Gevorkyan2021} that  
$$[x_0]=\bigcup_{n=1}^\infty G^n(x_0) $$
and hence, by assumption,
\begin{equation*}
X=\bigcup_{n=1}^\infty G^n(x_0).
\end{equation*}
Therefore, for any point $x\in X$ there is a natural number $n$ such that $x\in G^n(x_0)$.

Suppose that $x^*\in G^1(x_0)=G(x_0,x_0)$, i.e., there exists an element $g_0\in G$ such that $g_0(x_0,x_0)=x^*$. Then the map $\varphi:X \to X$ defined by
\begin{equation*}
    \varphi (x)=g_0(x_0,x)
\end{equation*} 
is a homeomorphism and translates $x_0$ to $x^*$: $\varphi(x_0)=g_0(x_0,x_0)=x^*$.

By induction, assume that for any point  $x\in G^{n}(x_0)$ one can find a homeomorphism 
$\varphi:X \to X$ such that $\varphi(x_0)=x$, and consider any point $x^*\in G^{n+1}(x_0)$. Since $G^{n+1}(x_0)=G(G^n(x_0),G^n(x_0))$, there are  $x', x''\in G^n(x_0)$ and $g'\in G$ such that 
$$x^*=g'(x', x'').$$
Because $x'' \in G^{n}(x_0)$, there is a  homeomorphism $\varphi': X\to X$ such that 
$$\varphi'(x_0)=x''.$$

Consider the map $\varphi'':X\to X$ defined by
\begin{equation*}
   \varphi'' (x)=g'(x',x),
\end{equation*}
for any $x\in X$, which is a homeomorphism.

Then, for the homeomorphism
$\varphi=\varphi''\circ \varphi':X\to X$
one has
$$
    \varphi(x_0)=\varphi''(\varphi'(x_0))=\varphi''(x'')=g'(x',x'')=x^*,
$$
which is what needed to be proven. 
\end{proof}

\begin{definition}\label{steps}
A homogeneous binary $G$-space $X$ is said to have \textit{the stabilization property at the $n$-th step at point} $x$, or alternatively, we say that a binary $G$-space $X$ \textit{stabilizes at the point $x$ at the $n$-step} if
$$G^{n}(x)=X \quad \text{but} \quad  G^{n-1}(x)\not=X,$$ 
where $G^n(x)$ is defined as in \eqref{eq-G^n}. 
\end{definition}%

In terms of Definition 4 in \cite{Gevorkyan2021}, Definition \ref{steps} divides finitely generated orbits in stabilization types. An important task in the study of homogeneous binary $G$-spaces would be to construct examples of each stabilization type.

In terms of Definition \ref{steps} the following is an easy result.

\begin{proposition}\label{distributive_step1}
Homogeneous distributive binary $G$-spaces have the stabilization property at step $1$ at any point.
\end{proposition}

\begin{proof}
As a homogeneous distributive binary $G$-space $X$ is transitive by Proposition \ref{prop_1}, one has $G^1(x)=G(x,x)= X$ for any point $x\in X$. 
\end{proof}

As it was seen before, the homogeneous binary $G$-spaces might contain points whose orbit is not the whole space (see Example \ref{ex-ZZZ}). It is also true that a homogeneous binary $G$-space can have different stabilization properties at different points. 

\begin{example}[Stabilization at different points at steps 1 and 2]
Consider the additive five element cyclic group $\mathbb{Z}_5$. The group of its invertible elements 
$G=\{1,2,3,4\}$ acts on this group by multiplication. 

Now consider the space $X=G$ with the binary $G$-action defined by the rule
\begin{equation}
    g(x,x')=g^xx'.
\end{equation}

It is easy to check that $G^1(1)=G(1,1)=X$. Therefore, the binary $G$-space $X$ stabilizes at step $1$ at the point $x=1$.

On the other hand, by direct computation we get $G^1(2)=G(2,2)=\{2,3\}$ and $G^2(2)=G(G^1(2), G^1(2))=X$. So, the binary $G$-space $X$ stabilizes at the point $x=2$ at step $2$. 
\end{example}

\begin{example}[Stabilization at step 3]
Consider the six-element symmetric group 
$S_3$ with the presentation
\begin{equation}
    S_3=\langle h,x : h^2=x^2=(xh)^3=e\rangle. 
\end{equation}
    
 It can be seen that
\begin{equation}
    S_3=\{e,x,h,xh,hx,xhx\}
\end{equation} and $xhx=hxh$.

Now consider the subgroup $G=\{e,h\}\subset S_3$ and define the binary action $\mu: G\times X^2 \to X$ of $G$ on $X=S_3$ by the rule
\begin{equation}\label{action_example}
    g(x,x')=x^{-1}gxx'.
\end{equation}

By direct computation, one has
\[
G^1(x)=G(x,x)=\{x,xh\},
\]
\[
G^2(x)=\{e,h,x,xh\},
\] 
\[
G^3(x)=X.
\]
Thus, the binary $G$-space $X$ stabilizes at the point $x$ at step 3.
\end{example}

Infinite binary $G$-spaces can stabilize at step $\infty $. 

\begin{example}[Stabilization at step $\infty$]
Let $G$ be an infinite group. Suppose that there exist $h,x \in G$ satisfying the following conditions:

(1) The elements $h$ and $x$ are of order 2: $h^2 = x^2 = e$, where $e$ is the identity element of $G$; 

(2) The element $xh\in G$ is of infinite order.

The subgroup $H = \{e, h\}$ of $G$ acts binarily on $G$ by the rule \eqref{action_example}. As shown in \cite[Theorem 3]{Gevorkyan2021}, the orbit of an element $x\in G$ is infinitely generated. Therefore, the orbit $[x]$, considered as a homogeneous binary $H$-space, stabilizes at the point $x$ at step $\infty$.
\end{example}

In the following we construct examples of stabilization at finite steps for infinite non-discrete binary $G$-spaces.

\begin{example}
Let $G=\mathbb{R}$ be the additive group of real numbers and let $X=\mathbb{R}^2$ be the plane.   Define a continuous map $\mu:G \times X^2 \to X$ by the formula 
\begin{equation}\label{example_step_2}
 g(\mathbf{x},\mathbf{x}')=(g\cdot \cos x_1+x_1',g\cdot \sin x_1+x_2' ),
\end{equation}
where $g\in G$, $\mathbf{x}=(x_1,x_2), \  \mathbf{x}'=(x'_1,x'_2) \in X$, and we have denoted $\mu(g, \mathbf{x},\mathbf{x}') = g(\mathbf{x},\mathbf{x}')$.

It can be verified that $\mu $ is a binary action of $G$ on $X$.

We will prove that this binary $G$-space is homogeneous and has the stabilization property at step $2$ at the point $\mathbf{x}_0=(0,0)$, i.e., $G^2(x_0)=X$.

First we compute $G^1(\mathbf{x}_0)=G(\mathbf{x}_0,\mathbf{x}_0)$. Consider an element $g\in G$. Then,
$$
g(\mathbf{x}_0,\mathbf{x}_0)=
(g\cdot \cos 0+0,g\cdot \sin 0+0 )=(g,0).
$$ 
So $G(\mathbf{x}_0,\mathbf{x}_0)=\mathbb{R}\times\{0\}\neq X $.

Now consider any point $\mathbf{x}=(x_1,x_2)\in X$ and take $g=\sqrt{x_1^2+x_2^2}$ and 
$x_1'=\tan^{-1}\left(\dfrac{x_2}{x_1} \right)$ if $x_1\neq 0$, and $x_1'=\pi/2$ if $x_1= 0$. Then one can check that
$$
g((x_1',0),(0,0))=\mathbf{x}
$$ 
and, therefore, $G^2(x_0)=X$.
\end{example}

This example can be generalized as follows.

\begin{example}
 Let $G=\mathbb{R}$ be the additive group of real numbers and let $X=\mathbb{R}^n$ be $n$-dimensional Euclidean space. Consider the continuous map 
$$\mu:G \times X^2 \to X, \quad \mu(g,\mathbf{x}, \mathbf{y})=g(\mathbf{x}, \mathbf{y}) = \mathbf{z} =(z_1, \dots , z_n)
$$
given by the formulas 
\begin{equation}\label{example_step_n}
\begin{array}{ll}
z_1= g\cdot \cos x_1 +y_1, &\\
z_2 = g\cdot \sin x_1 \cdot \cos x_2 +y_2,&\\
z_3 = g\cdot \sin x_1\cdot \sin x_2\cdot \cos x_3+y_3,&\\
\dots &\\
z_{n-1} = g\cdot \sin x_1 \cdot \sin x_2 \cdot \ldots\cdot \sin x_{n-2} \cdot \cos x_{n-1} +y_{n-1},&\\
z_n = g\cdot \sin x_1 \cdot \sin x_2\cdot \ldots\cdot \sin x_{n-2}\cdot \sin x_{n-1} +y_{n}.
\end{array}
\end{equation}
for any $g\in G$, $\mathbf{x}=(x_1,\dots, x_n), \mathbf{y}=(y_1,\dots, y_n)\in \mathbb{R}^n$.

The proof that the map $\mu$ is a binary action of the group $G$ on $X$ we leave to the reader. 

Let us prove that this binary $G$-space is homogeneous and has the stabilization property at the $n$-th step at the point $\mathbf{x}_0= (0,\dots,0)$, i.e., 
\begin{equation}\label{eq-GnRn}
G^n(\mathbf{x}_0)=\mathbb{R}^n.
\end{equation}
This will be proved by induction, using hyperspherical coordinates in an $n$-dimensional Euclidean space $\mathbb{R}^n$.

First, we compute $G^1(\mathbf{x}_0)=G(\mathbf{x}_0,\mathbf{x}_0)$. For any element $g\in G$, by \eqref{example_step_n}, we have
\begin{multline*}
g(\mathbf{x}_0,\mathbf{x}_0)=(g\cdot \cos 0+0,g\cdot \sin 0\cdot \cos 0+0,g\cdot \sin 0\cdot \sin 0\cdot \cos 0+0, \dots\\
\dots,g\cdot \sin 0\cdot \ldots \cdot \sin 0\cdot \cos 0+0,
g\cdot \sin 0\cdot \ldots \cdot \sin 0+0 )=(g,0,\dots, 0).
\end{multline*}
So, $G^1(\mathbf{x}_0)=\mathbb{R}\subset \mathbb{R}^n$.

Now assume that 
$G^{k-1}(\mathbf{x}_0)=\mathbb{R}^{k-1}$ and then prove that $G^k(\mathbf{x}_0)= \mathbb{R}^k$. For this, it suffices to prove that
$$
G(G^{k-1}(\mathbf{x}_0),\mathbf{x}_0)
=\mathbb{R}^k.
$$ 
Consider any point $\mathbf{x}=(x_1,x_2, \dots , x_{k-1}, 0, \cdots , 0)\in \mathbb{R}^{k-1}$, and denote $g(\mathbf{x}, \mathbf{x}_0) = \mathbf{z} =(z_1, \dots , z_n)$. By \eqref{example_step_n}, we obtain

$
\begin{array}{ll}
z_1= g\cdot \cos x_1, &\\
z_2 = g\cdot \sin x_1 \cdot \cos x_2,&\\
z_3 = g\cdot \sin x_1\cdot \sin x_2\cdot \cos x_3,&\\
\dots &\\
z_{k-1} = g\cdot \sin x_1 \cdot \sin x_2 \cdot \ldots\cdot \sin x_{k-2} \cdot \cos x_{k-1},&\\
z_{k} = g\cdot \sin x_1 \cdot \sin x_2 \cdot \ldots\cdot \sin x_{k-1},&\\
z_{k+1} = 0,&\\
\dots &\\
z_n = 0,
\end{array}
$

\noindent
which is the expression of an arbitrary point 
$(z_1, \dots , z_k)\in \mathbb{R}^{k}\subset \mathbb{R}^{n}$ in hyperspherical coordinates $(g, x_1, \dots, x_{k-1})$. This means that
$G(G^{k-1}(\mathbf{x}_0),\mathbf{x}_0)
=\mathbb{R}^{k}$.
\end{example}

In view of the previous examples, one has the following. 

\begin{theorem}
   For any natural $n\in \mathbb{N}$
   there exist a (non-discrete, Abelian) topological group $G$ and a (non-discrete) homogeneous binary $G$-space $X$ with the stabilization property at step $n$. 
\end{theorem}


\end{document}